\documentclass[12pt,reqno]{amsart}
\usepackage{amsmath,amssymb}

\usepackage{hyperref}

\usepackage{graphicx,psfrag}

\newcommand{\vectornorm}[1]{\left|\left|#1\right|\right|}
\newcommand{\R}{\mathbb{R}}

\newcommand{\varep}{\varepsilon}

\newtheorem{theorem}{Theorem}[section]
\newtheorem{lemma}[theorem]{Lemma}

\newtheorem{proposition}[theorem]{Proposition}

\newtheorem{definition}[theorem]{Definition}

\numberwithin{equation}{section}

\begin{document}

\title[Uniqueness and stability of saddle-shaped solutions]
{Uniqueness and stability of saddle-shaped solutions to the Allen-Cahn equation}

\author{Xavier Cabr{\'e}}
\thanks{The author was supported by
grants MTM2008-06349-C03-01 (Spain) and 2009SGR-345 (Catalunya)}

\address{ICREA and Universitat Polit{\`e}cnica de Catalunya,
Departament de Matem{\`a}tica Aplicada I, Diagonal 647, 08028
Barcelona, Spain}
\email{xavier.cabre@upc.edu}

\begin{abstract}
We establish the uniqueness of a saddle-shaped 
solution to the diffusion equation $-\Delta u = f(u)$ 
in all of $\mathbb{R}^{2m}$, where $f$ is of bistable type, in every even dimension
$2m \geq 2$. In addition, we prove its stability whenever $2m \geq 14$. 

Saddle-shaped solutions are odd with respect to the Simons cone 
${\mathcal C} = \{ (x^1,x^2) \in \mathbb{R}^m \times \mathbb{R}^m : |x^1|=|x^2| \}$ 
and exist in all even dimensions. Their uniqueness was only known 
when $2m=2$. On the other hand, they are known to be unstable in 
dimensions 2, 4, and~6. Their stability in dimensions 8, 10, and 12 remains an open question.
In addition, since the Simons cone minimizes area when $2m \geq 8$,
saddle-shaped solutions are expected to be global minimizers when $2m \geq 8$,
or at least in higher dimensions. This is a property stronger than stability which is
not yet established in any dimension.
\end{abstract}

\maketitle

\section{Introduction}

This paper concerns saddle-shaped solutions
to bistable diffusion equations
\begin{equation}\label{eq}
-\Delta u=f(u)\quad {\rm in }\,\R^{2m},
\end{equation}
where $2m$ is an even integer. Throughout the paper we assume that $f$ is a $C^{2,\alpha}$
function on $[-1,1]$, for some $\alpha\in (0,1)$, such that
\begin{equation}\label{hypf}
f \text{ is odd, }\; f(0)=f(1)=0,\text{ and }
f''<0 \text{ in } (0,1).
\end{equation}
As a consequence we have $f>0$ in $(0,1)$. Under these assumptions, we say that
$f$ is of bistable type. A typical example is the nonlinearity
$f(u)=u-u^3$ in the Allen-Cahn equation.

For $x=(x_1,\dots, x_{2m})\in\R^{2m}$, consider the two
radial variables
\begin{equation*}\label{coor}
\left\{\begin{array}{rcl} 
s & = & (x_1^2+...+x_m^2)^{1/2}\\
\vspace{-3mm}\\
t & = & (x_{m+1}^2+...+x_{2m}^2)^{1/2}.
\end{array}
\right.
\end{equation*}
A {\it saddle-shaped solution} of \eqref{eq} is a solution $u$ of \eqref{eq} which
depends only on $s$ and $t$, satisfies $|u|<1$, is positive in $\{s>t\}$,
and is odd with respect to  $\{s=t\}$, i.e., $u(t,s)=-u(s,t)$.

Consider also the variables
\begin{equation*}\label{defyz}
\left\{\begin{array}{rcl}
y & = & (s+t)/\sqrt{2}\\
z & = & (s-t)/\sqrt{2},
\end{array}\right.
\end{equation*}
which satisfy $y\geq 0$ and $-y\leq z\leq y$.
An important set in what follows is the
Simons cone (see Figure \ref{fig1}), defined by 
$$
{\mathcal C}=\{s=t\}=\{z=0\}=\partial {\mathcal O},
$$
where
$$
{\mathcal O}=\{s>t\}=\{z>0\}\subset\R^{2m}.
$$ 
The cone ${\mathcal C}$ has zero mean curvature at every
$x\in{\mathcal C}\backslash\{0\}$, in every dimension $2m\geq 2$.
However, it is only in dimensions $2m\geq 8$ that ${\mathcal C}$ is in addition a
minimizer of the area functional. Furthermore,  ${\mathcal C}$ is stable
only in these same dimensions; see \cite{CT2} and references therein.
We will see that similar properties hold, or are expected to hold,
for saddle-shaped solutions of \eqref{eq}.

\begin{figure}[htp]
	\begin{center}
		\psfrag{t}{$t \ge 0$}
		\psfrag{s}{$s \ge 0$}
		\psfrag{dy}{$\partial_y$}
		\psfrag{dz}{$\partial_z$}
		\psfrag{C}{${\mathcal C}=\{s=t\}=\{z=0\}$}
		\psfrag{O}{${\mathcal O}=\{s>t\}=\{z>0\}$}
		\includegraphics[width=4cm]{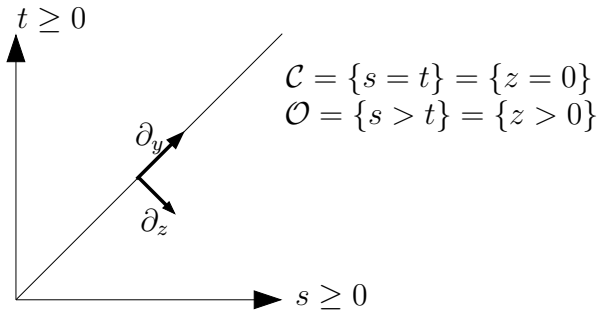}
	\end{center}
	\caption{The Simons cone ${\mathcal C}$. The $(s,t)$ and $(y,z)$ variables}
	\label{fig1}
\end{figure}

Under our assumptions \eqref{hypf} on $f$, there exists a unique increasing solution of
$-\Delta u=f(u)$ in all of $\R$ up to translations of the
independent variable; see, e.g., Lemma~4.3 of \cite{CT1}. It has limits $\pm 1$ at $\pm\infty$. 
We normalize it to vanish at the origin and
we call it~$u_0$. Thus, we have
\begin{equation*}\label{u0}
\left\{
\begin{array}{l}
u_0:\R\rightarrow (-1,1),\; -\ddot{u}_0=f(u_0) \; \text{ in } \R ,\\
u_0(0)=0,\; \dot{u}_0>0 \ \text{ in } \R, \;\text{ and }
\; u_0(\tau)\to\pm 1  \ \text{ as }\tau\to\pm\infty.
 \end{array}
 \right.
\end{equation*}
For the Allen-Cahn nonlinearity $f(u)=u-u^3$, 
the solution $u_0$ can be computed explicitly and it is given by 
$u_0(\tau)=\tanh(\tau/\sqrt{2})$.

It is simple to check that $|z|$ is the distance in $\R^{2m}$ from any
point $x\in \R^{2m}$ to the Simons cone ${\mathcal C}$; see Lemma~4.2 of \cite{CT1}. 
This is important when showing that the function
\begin{equation}\label{defU}
U(x):=u_0\left( \frac{s-t}{\sqrt{2}}\right)=u_0(z) \qquad\text{for }x\in\R^{2m}
\end{equation}
describes the asymptotic behavior
of saddle-shaped solutions at infinity.
This was established by J.~Terra and the author \cite{CT2} and it is
stated in Theorem~\ref{exis-uns} below.
Note that $U$ is a Lipschitz function in $\R^{2m}$, but it is not differentiable
at $\{st=0\}$. Using \eqref{eqst}, which is equation \eqref{eq} expressed in the 
$(s,t)$ variables, one sees that $-\Delta U>f(U)$
in $\{s>t>0\}$, i.e., $U$  is a strict supersolution 
in $\{s>t>0\}$ ---a fact that we will not use in this paper.

The energy functional associated to equation \eqref{eq} is
\begin{equation}\label{energia}
{\mathcal E}(u,\Omega):=\int_\Omega \left\{ \frac{1}{2} |\nabla u|^2
+G(u)\right\} dx,\qquad\text{where } \, G'=-f.
\end{equation}
We say that a bounded solution $u$ of
\eqref{eq} is {\it stable} if the second variation of energy
$\delta^2{\mathcal E}/\delta^2\xi$ with respect to compactly
supported smooth perturbations $\xi$ is nonnegative. That is, if
\begin{equation}\label{stable}
Q_u(\xi):=\int_{{\R}^{2m}}\left\{
|\nabla\xi|^2-f'(u)\xi^2\right\}dx\geq 0 \quad {\rm for \, all}
\;\xi\in C^\infty_c(\R^{2m}).
\end{equation}
We say that $u$ is {\it unstable} when $u$ is not stable.
The stability of $u$ is equivalent to requiring the linearized
operator $-\Delta -f'(u)$ to have positive first eigenvalue (or to
satisfy the maximum principle) in every smooth bounded domain
of $\R^{2m}$; see section~2 for these questions.

A bounded function $u\in C^1(\R^{2m})$ is said to be a {\it global minimizer} of \eqref{eq}
when ${\mathcal E}(u,\Omega)\leq {\mathcal E}(v,\Omega)$ for every bounded
domain $\Omega$ and function $v\in C^1(\overline\Omega)$ such that $v\equiv u$ on 
$\partial \Omega$. Clearly, every global minimizer is a stable solution.

The following results were proven in \cite{CT1,CT2} by J.~Terra and the author. 

\begin{theorem}[\textbf{Cabr\'e-Terra \cite{CT1,CT2}}]\label{exis-uns} 
Assume that $f$ satisfies \eqref{hypf}.

{\rm (a)} For every even dimension  $2m\geq 2$,
there exists a saddle-shaped solution $u\in C^2(\R^{2m})$ 
of $-\Delta u= f(u)$ in~$\R^{2m}$, that is, a solution $u$ of this equation
with $|u|<1$ and such that 
\begin{itemize}
\item $u$ depends only on the variables $s$ and $t$. We write
$u=u(s,t)$; 
\item $u>0$ in ${\mathcal O}=\{s>t\}$; 
\item  $u(t,s)=-u(s,t)$ in $\R^{2m}$.
\end{itemize}

{\rm (b)} For $2m\geq 2$, every saddle-shaped solution $u$ satisfies
\begin{equation}
\label{unif}
\vectornorm{\, |u-U| + |\nabla (u-U)|\, }_{L^{\infty}(\R^{2m}\setminus B_R(0))} 
\rightarrow 0\quad
\text{as }\, R\rightarrow\infty ,
\end{equation}
where $U$ is defined by \eqref{defU}.

{\rm (c)} When $2\leq 2m\leq 6$, every saddle-shaped solution
$u$ is unstable. Furthermore, in dimensions $4$ and $6$ every saddle-shaped solution
has infinite Morse index in the sense of Definition~{\rm 1.8} of \cite{CT2}.
\end{theorem}

Part (a) of the theorem concerns existence. It can be established
in different ways, all of them rather simple ---for instance, variationally as in 
section~3 of \cite{CT1}, or through monotone iteration
as in section~3 of \cite{CT2}.

Part (b) gives the asymptotic behavior of saddle-shaped solutions at infinity.
In this paper we will also need the asymptotics of second derivatives. 
See Lemma~\ref{lemmast} below,
where we sketch also the proof of \eqref{unif}.
 
The uniqueness of a saddle-shaped solution was only known in dimension 2,
even for the Allen-Cahn nonlinearity. This was established by Dang,
Fife, and Peletier \cite{DFP}, who also proved ---in dimension 2---
existence of a saddle-shaped solution and some results on its asymptotic behavior 
and its monotonicity properties.

The instability of the saddle-shaped solution in $\R^2$, indicated in a partial
result of \cite{DFP}, was studied in detail by Schatzman~\cite{Sc} by analyzing the linearized
operator at the saddle-shaped solution and showing that, when $f(u)=u-u^3$, this
operator has exactly one negative eigenvalue. That is, the saddle-shaped solution 
of the Allen-Cahn equation in $\R^2$ has Morse index 1 ---in contrast
with our result Theorem~\ref{exis-uns}(c) in dimensions 4 and 6. 
See \cite{CT2} for more comments on this.

Our first main result is the following uniqueness theorem.

\begin{theorem}\label{unique}
Assume that $f$ satisfies \eqref{hypf}. Then, for every even dimension  $2m\geq 2$,
there exists a unique saddle-shaped solution $u$ 
of $-\Delta u= f(u)$ in~$\R^{2m}$.
\end{theorem}

The uniqueness result will follow (see section~3) from two main ingredients: 
the asymptotics \eqref{unif} at infinity for saddle-shaped solutions 
and the following new result ---a maximum principle in ${\mathcal O}=\{s>t\}$
for the linearized operator at every saddle-shaped solution.

\begin{proposition}\label{mpsaddle}
Assume that $f$ satisfies \eqref{hypf}.
Let $u$ be a saddle-shaped solution of \eqref{eq}, where $2m\geq 2$. 

Then, the maximum principle holds for the operator 
$$
L_u:=\Delta + f'(u(x)) \quad\text{ in } {\mathcal O}=\{ s>t\},
$$
in the sense that whenever $v \in C^{2}({\mathcal O})\cap C(\overline{\mathcal O})$ satisfies
\begin{equation}\label{hypintroMP}
L_u v \geq 0 \hbox{ in } {\mathcal O}, \quad v\le 0 \hbox{ on
} \partial{\mathcal O},\hbox{\quad and \quad} \limsup_{x\in{\mathcal O}, |x|\to\infty}
v(x)\le 0,
\end{equation}
then necessarily $v \le 0$ in ${\mathcal O}$.
\end{proposition}

We establish this result in section~2 using as key ingredient 
a maximum principle in the ``narrow" domain $\{t<s<t+\varepsilon\}$,
where $\varepsilon$ is small; see Lemma~\ref{mpnarrow} below.

As we explain at the end of this introduction, due to a connection between 
minimal surfaces and solutions of the Allen-Cahn equation,
saddle-shaped solutions are expected to be global minimizers of \eqref{eq}
(as defined in the beginning of this introduction) in dimensions 8 and higher
---or at least in dimensions high enough. This is still an open question ---already raised
in 2002 by Jerison and Monneau; 
see conjecture C4 and section~1.3 in~\cite{JM}.

Towards the understanding of this minimality property, in this article we establish 
the stability of saddle-shaped solutions in dimensions $14$ and higher.
This is our second main result. Of course, stability is a weaker property than minimality.
Stability in dimensions 8, 10, and 12 remains an open question.

To state the result, we need to choose a number $b$ such that
\begin{equation}\label{ineqb}
b(b-m+2)+m-1\leq 0.
\end{equation}
This is possible only when $m\geq 7$,
in which case one can take any 
\begin{equation}\label{rangeb}
b\in [b_-,b_+],\quad\text{ where } b_{\pm}=
\frac{m-2}{2}\pm \frac{\sqrt{(m-2)^2-4(m-1)}}{2}.
\end{equation}
It follows that $b>0$.

\begin{theorem}\label{stab} 
Assume that $f$ satisfies \eqref{hypf}.
If $2m\geq 14$, the saddle-shaped solution $u$ of \eqref{eq}
is stable in $\R^{2m}$, i.e., \eqref{stable} holds. 
Furthermore, for every $b>0$ satisfying \eqref{rangeb}, 
the function
\begin{equation}\label{dirder}
\varphi := t^{-b}u_s -s^{-b}u_t
\end{equation}
is of class $C^2$ and positive in $\R^{2m}\setminus\{st=0\}$, 
it is even with respect to the Simons cone,
and it is a supersolution of the linearized equation
\begin{equation*}
\Delta \varphi + f'(u) \varphi \leq 0 \quad\text{ in } \R^{2m}\setminus\{st=0\}.
\end{equation*}
\end{theorem}

The statement on the stability of the saddle-shaped solution will follow
immediately from the properties of $\varphi$ stated in the theorem.
The key ingredients to establish Theorem~\ref{stab} are the following monotonicity
and convexity properties of saddle-shaped solutions. They hold in every even dimension.
 
\begin{proposition}\label{monot}  
Assume that $f$ satisfies \eqref{hypf}.
Let $u$ be the saddle-shaped solution of \eqref{eq}, where $2m\geq 2$. 
We then have
\begin{eqnarray}
\label{signuy}
 u_y >0 & & \text{in } {\mathcal O}=\{ s>t\},\\
\label{signut}
 -u_t >0 & & \text{in } {\mathcal O}\setminus\{t=0\}=\{ s>t>0\},
\end{eqnarray}
and
\begin{equation}
\label{second}
 u_{st} >0 \qquad \text{in } {\mathcal O}\setminus\{t=0\}=\{ s>t>0\}.
\end{equation}
\end{proposition}

The cone of monotonicity
generated by $\partial_y$ and $-\partial_t$ is the optimal one (i.e., the largest one)
holding at all points of $\{ s>t>0\}$. In Figure~\ref{fig2} we draw this cone,
and also the shape of the level sets of a saddle-shaped solution ---where we
take into account the asymptotic result \eqref{unif}.

\begin{figure}[htp]
	\begin{center}
		\psfrag{t}{$t \ge 0$}
		\psfrag{s}{$s \ge 0$}
		\psfrag{C}{${\mathcal C}$}
		\psfrag{u}{$\mu$}
		\psfrag{dy}{$\partial_y$}
		\psfrag{dt}{$-\partial_t$}
		\psfrag{l}{$\{u=\lambda=u_0(\mu)\}$}
		\includegraphics[width=5cm]{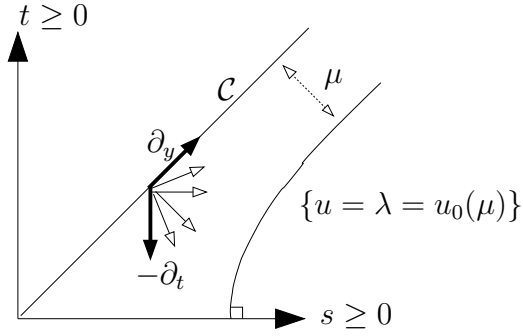}
	\end{center}
	\caption{The cone of monotonicity and the level sets of the saddle-shaped solution}
	\label{fig2}
\end{figure}

The monotonicity properties \eqref{signuy} and \eqref{signut} were first
proven by J.~Terra and the author in \cite{CT2} for the so-called maximal 
saddle-shaped solution ---at that point uniqueness of saddle-shaped solution
was not known. Here we establish these properties using
a different method. 

The second derivative
property \eqref{second} is a new fact proved in this paper. It is a crucial ingredient to
establish stability in dimensions 14 and higher.

The three inequalities in Proposition~\ref{monot} will be proven using 
our maximum principle for the 
linearized operator in ${\mathcal O}$, or rather a slightly more general version:
Proposition~\ref{mpholds} of next section.

To establish such maximum principle, note that since $f(0)=0$ and $f''<0$ in $(0,1)$, 
we have $f(\rho)/\rho > f'(\rho)$
for all $0<\rho<1$. Thus, if $u$ is a saddle-shaped solution then
\begin{equation}\label{possuper}
-\Delta u =f(u) > f'(u) u \qquad\text{in } {\mathcal O}.
\end{equation}
That is, $u$ is a positive and strict supersolution of the linearized
operator $\Delta +f'(u)$ at $u$ in all ${\mathcal O}$.
This will be one of the ingredients (but not the only one since $\inf_{{\mathcal O}}u=0$)
to establish the maximum principle 
in~${\mathcal O}$ for the linearized operator. The other ingredient will be
a maximum principle for the linearized operator 
in the ``narrow" domain $\{t<s<t+\varepsilon\}$, Lemma~\ref{mpnarrow}.

To prove Proposition~\ref{monot}, we will apply the maximum principle of
Proposition~\ref{mpholds} to the equations
satisfied by $u_s$, $u_t$, $u_y$, and $u_{st}$. For this, note that equation \eqref{eq}
in the $(s,t)$ variables reads
\begin{equation}\label{eqst}
u_{ss}+u_{tt}+(m-1){\Big (}\frac{u_s}{s}+\frac{u_t}{t}{\Big
)}+f(u)=0 
\end{equation}
for $s>0$ and $t>0$,
while in the $(y,z)$ variables it becomes
\begin{equation}\label{eqyz}
u_{yy}+u_{zz}+\frac{2(m-1)}{y^2-z^2}(yu_y-zu_z)+f(u)=0
\end{equation}
for $|z|<y$.

Finally, let us explain why $2m\leq 6$ is relevant in Theorem~\ref{exis-uns}(c)  
and why stability and minimality are expected in higher dimensions.
Due to a relation between the Allen-Cahn equation
and the theory of minimal surfaces (see~\cite{Sa,JM,dP,AAC}), every level set
of a global minimizer of \eqref{eq} should converge at infinity 
in some weak sense
to the boundary of a minimal set ---minimal in
the variational sense, that is, minimizing perimeter. 
Note now that the zero level set of a saddle-shaped solution is the Simons
cone ${\mathcal C}$. It is easy to
verify that ${\mathcal C}$ has zero mean curvature at every
$x\in{\mathcal C}\backslash\{0\}$, in every dimension $2m\geq 2$.
However, it is only in dimensions $2m\geq 8$ when ${\mathcal C}$ is in addition a
minimizer of the area functional, i.e., it is a minimal cone 
in the variational sense.  Furthermore,  ${\mathcal C}$ is stable
only in these same dimensions. See \cite{CT2} and references therein for these
questions.

Furthermore, a deep theorem states that the boundary of
a minimal set in all of $\R^n$ must be a hyperplane if $n\leq 7$.
Instead, in $\R^8$ and higher dimensions, there exist minimal sets different
than half-spaces ---the simplest example being the Simons cone.

The analogue for the Allen-Cahn equation of the first of these two results is well understood.
Indeed, a deep theorem of Savin~\cite{Sa} states that in dimensions $n\leq 7$, 
$1$D solutions (i.e., solutions depending only on one Euclidean variable)
are the only global minimizers of the Allen-Cahn equation. Note that this result makes 
no assumption on the monotonicity or limits at infinity of the solution. That is:

\begin{theorem}[\textbf{Savin \cite{Sa}}]\label{teosavin}
Assume that $n\leq 7$ and that $u$ is a global minimizer of 
$-\Delta u=u-u^3$ in $\R^n$. 
Then, the level sets of $u$ are hyperplanes.
\end{theorem}

However, the analogue for the Allen-Cahn equation of the second statement 
(i.e., the minimality of the Simons cone when  $2m\geq 8$)
is not yet understood. That is, the possible minimality in $\R^8$
(or at least in higher dimensions)
of the saddle-shaped solution is still unknown.
This question was already raised in 2002 by Jerison and Monneau~\cite{JM}.

\medskip

\noindent
\textbf{Open Question 1.7.}
Is the saddle-shaped solution a
global minimizer of the Allen-Cahn equation in $\R^{2m}$ for $2m\geq 8$, 
or at least in higher even dimensions?

\medskip
\addtocounter{theorem}{1}

Related to this, in $\R^9$ it is known the existence of a global minimizer
to the Allen-Cahn equation which is not $1$D (i.e., with
level sets different than hyperplanes). This is the solution in $\R^9$,
monotone in the $x_9$ variable, 
constructed by del Pino, Kowalczyk, and Wei~\cite{dP}.
Since this monotone solution is known to have limits $\pm 1$ as $x_{9}\to\pm\infty$,
a result of Alberti, Ambrosio, and the author \cite{AAC} guarantees that the
solution is indeed a global minimizer.

In this direction, a positive answer to the Open Question~1.7 above would 
give an alternative way 
to that of \cite{dP} to prove the existence of a counter-example
to the conjecture of De Giorgi in $\R^9$. Indeed, saddle-shaped
solutions are even functions of each coordinate $x_i$. Thus,
by a result of Jerison and Monneau~\cite{JM}, if the saddle-shaped solution
were a global minimizer in $\R^{2m}$, then the conjecture of De Giorgi
on monotone solutions would not hold in $\R^{2m+1}$.
Indeed, from the 1D solution depending only on $x_{2m+1}$ and from a 
global minimizing saddle-shaped solution 
depending only on $(x_1,\ldots,x_{2m})$,
\cite{JM} constructs in a natural way
a solution in $\R^{2m+1}$ which is monotone in
the last variable $x_{2m+1}$ and which is not 1D. However, it is not proved
that the solution of \cite{JM}
has limits $\pm 1$ as $x_{2m+1}\to\pm\infty$
---a property known for the solution in $\R^9$ of \cite{dP}.

The plan of the paper is the following.
In section~2 we prove the maximum principle for the linearized operator
in ${\mathcal{O}}$. Section~3 establishes the uniqueness of saddle-shaped solution,
Theorem~\ref{unique}.
In section~4 we establish Proposition~\ref{monot} 
on the sign of derivatives of $u$. Finally, section~5 concerns the supersolution
of the linearized equation and completes the proof of our
stability result, Theorem~\ref{stab}.

\section{The maximum principle for $\Delta+f'(u)$ in $\mathbb{\mathcal O}$}

Let us consider linear operators of the form
\begin{equation*}
Lv:= \Delta v + c(x) v,
\end{equation*}
where $c$ is a continuous function (perhaps unbounded) in an open set $\Omega$ of $\R^n$,
i.e., $c \in C(\Omega)$.

\begin{definition}\label{defmp}
{\rm
Let $\Omega$ be an open set of $\R^n$. 
We say that the maximum principle holds for the operator $L$ in
$\Omega$ if, whenever $v \in C^{2}(\Omega)\cap
C(\overline\Omega)$ satisfies
\begin{equation}\label{hypMP}
Lv \geq 0 \hbox{ in } \Omega, \quad v\le 0 \hbox{ on
} \partial\Omega,\hbox{\quad and \quad} \limsup_{x\in\Omega, |x|\to\infty}
v(x)\le 0,
\end{equation}
then necessarily $v \le 0$ in $\Omega$.
}
\end{definition}

The last condition in \eqref{hypMP} only plays a role when $\Omega$ is unbounded.

The main result in this section is the following.

\begin{proposition}\label{mpholds}
Assume that $f$ satisfies \eqref{hypf}.
Let $u$ be a saddle-shaped solution of \eqref{eq}, where $2m\geq 2$. 
Let $\Omega\subset {\mathcal O}=\{ s>t\}$ be an open set and $c\in C(\Omega)$ with
$c\leq 0$ in~$\Omega$. 

Then, the maximum principle holds for the operator 
$L_u+c(x)=\Delta +\{ f'(u(x))+c(x) \}$ in~$\Omega$.
\end{proposition}

We will use this result to prove both Theorem~\ref{unique} on uniqueness and 
Proposition \ref{monot} on the sign of $u_y$, $u_t$, and 
$u_{st}$. For this we will use the previous proposition both with
$\Omega={\mathcal O}=\{ s>t\}$ and with $\Omega={\mathcal O}\setminus\{t=0\}=\{ s>t>0\}$.
We will need to use it with different choices of coefficient $c=c(x)$,
with $c$ continuous and nonpositive in $\Omega$ but unbounded below.

The rest of this section is devoted to prove Proposition~\ref{mpholds}. For this,
recall that the typical way towards establishing the maximum principle 
for an operator $L$ in an open set $\Omega$ is to first show that
\begin{equation}\label{knowexist}
\text{``there exists a positive supersolution $\phi$ of $L\phi=0$ in $\Omega$''.}  
\end{equation}
If this holds, then adding one of various additional
assumptions on $\phi$ (the simplest one being $\phi\geq c>0$ with $c$ 
a positive constant), it does guarantee the maximum principle to
hold; see \cite{BNV}. Indeed, in bounded domains, \eqref{knowexist} is a 
necessary ---and ``almost'' sufficient--- condition for the maximum principle to
hold; see Corollary 2.1 of \cite{BNV}. However, in unbounded domains one has to
be more careful to deal with infinity.

Recall that when $u$ is a saddle-shaped solution of \eqref{eq},
$u$ is a positive supersolution of the linearized
operator $\Delta +f'(u)$ at $u$ in all ${\mathcal O}$;
see \eqref{possuper}. However, $\inf_{\mathcal O} u=0$ since $u=0$ on $\partial {\mathcal O}$.
Nevertheless, we claim that for every given $\varepsilon>0$, we have that
\begin{equation}\label{positu}
u \geq \delta >0 \qquad\text{in } {\mathcal O}_\varepsilon:=\{s>t+\varepsilon\}
\end{equation}
for some positive constant $\delta$ (which may depend on 
the particular solution $u$). Indeed,
$U(x)=u_0(z)\geq u_0(\varepsilon /\sqrt{2})>0$ in ${\mathcal O}_\varepsilon$.
Hence, by the asymptotic behavior of saddle-shaped solutions at infinity,
Theorem~\ref{exis-uns}(b),
there exists a radius $R>0$ such that 
$u(x)\geq u_0(\varepsilon /\sqrt{2})/2$ (a positive constant)
if $|x|>R$ and $x\in {\mathcal O}_\varepsilon$. Now, since $u$ is positive in the compact set  
$\overline{\mathcal O}_\varepsilon\cap \overline{B}_R(0)$, we conclude the claim.

The lower bound \eqref{positu} in ${\mathcal O}_\varepsilon$ together with the
following maximum principle in ${\mathcal N}_\varepsilon :=
{\mathcal O}\setminus \overline{\mathcal O}_\varepsilon$,
a ``narrow'' domain in a sense explained later,
will lead to the maximum principle for $L_u$ in all of~${\mathcal O}$.

\begin{lemma}\label{mpnarrow}
Let $2m\geq 2$, $\varepsilon >0$, and 
$$
{\mathcal N}_\varepsilon := \{ t<s<t+\varepsilon\}\subset \R^{2m}.
$$
Let $H\subset {\mathcal N}_\varepsilon$ be an open set and
$\tilde{c}\in C(H)$ satisfy $\tilde{c}_+\in L^\infty(H)$, where
$\tilde{c}_+$ denotes the positive part of $\tilde{c}$.

Then, the maximum principle holds
for the operator $\Delta +\tilde{c}(x)$ in $H$ whenever 
\begin{equation}\label{smallep}
C_m\; \varep^2 \; \|\tilde{c}_+\|_{L^\infty(H)} < 1,
\end{equation}
where $C_m$ is a positive constant depending only on~$m$.
\end{lemma}

The constant $C_m$ can be taken to be equal to $3$ for all $m$ ---this will
be seen below, in our second proof of the lemma.
Note that in this result $\tilde{c}$ is allowed to change sign ---in contrast with
Proposition~\ref{mpholds}. 

Using Lemma~\ref{mpnarrow} we can now prove
Proposition~\ref{mpholds}.

\begin{proof}[Proof of Proposition~\ref{mpholds}]
Let $u$ be a saddle-shaped solution of \eqref{eq} and let
$$
Lv:=L_u v+c(x)v=\Delta v+ \{f'(u(x))+c(x)\} v.
$$ 
Since $c\leq 0$ in~$\Omega\subset{\mathcal O}$, then $f'(u)+c\leq f'(u)\leq 
\max_{[0,1]}f'=f'(0)$ in~$\Omega$.
Choosing $\varepsilon =(2C_m f'(0))^{-1/2}$, Lemma~\ref{mpnarrow} states that
the maximum principle holds
for the operator $L$ in any open subset of ${\mathcal N}_\varepsilon$.

Let 
$$
\Omega_\varepsilon := \Omega\cap{\mathcal O}_\varepsilon= \Omega\cap
\{ s>t+\varepsilon \}
$$
and
$$
H_\varepsilon := \Omega\cap{\mathcal N}_\varepsilon= \Omega\cap
\{ t<s<t+\varepsilon \}\subset {\mathcal N}_\varepsilon.
$$
Note that
\begin{equation}\label{boundomega}
\partial \Omega_\varepsilon \subset \partial\Omega \cup (\Omega\cap \{ s=t+\varepsilon \} ),
\end{equation}
\begin{equation}\label{boundaries}
\partial H_\varepsilon \subset \partial\Omega \cup (\Omega\cap \{ s=t+\varepsilon \} )
\subset \partial\Omega \cup \overline{\Omega}_\varepsilon 
\end{equation}
since $\Omega\cap\{s=t\}=\varnothing$, and
\begin{equation}\label{full}
\overline\Omega = \overline{\Omega}_\varepsilon \cup \overline{H}_\varepsilon. 
\end{equation}

Recall that $u>0$ in $\Omega\subset {\mathcal O}$ and that by 
\eqref{positu} we know that
\begin{equation}\label{posstr}
u\geq \delta >0 \quad\text{in } \Omega_\varepsilon
\end{equation}
for some constant $\delta>0$. In addition, by \eqref{possuper},
\begin{equation}\label{stricL}
Lu=\Delta u +\{f'(u)+c\}u\leq \Delta u +f'(u)u < 0\quad\text{ in } \Omega .
\end{equation}

Let $v \in C^{2}(\Omega)\cap
C(\overline\Omega)$, as in the definition of the maximum principle, satisfy
\begin{equation}\label{hypv2}
Lv \geq 0 \hbox{ in } \Omega, \quad v\le 0 \hbox{ on
} \partial\Omega,\hbox{\quad and \quad} \limsup_{x\in\Omega, |x|\to\infty} v(x)\le 0.
\end{equation}
Consider
$$
w:= \frac{v}{u} \quad \text{ in } \Omega .
$$
By \eqref{posstr} and the hypotheses \eqref{hypv2} on $v\in C(\overline\Omega)$,
$w$ is bounded above in~$\Omega_\varepsilon$. 

Assume that
\begin{equation}\label{contrad}
S:=\sup_{\overline{\Omega}_\varepsilon} w >0.
\end{equation}
Then, by the two last conditions in \eqref{hypv2} and by \eqref{boundomega}, 
this supremum must be achieved at a point $x_0\in \Omega_\varepsilon \cup 
(\Omega\cap\{ s=t+\varepsilon \})\subset \Omega$.

We have that $v-Su\leq 0$ in $\overline{\Omega}_\varepsilon$. Therefore,
$v-Su$ is a subsolution for $L$ in $H_\varepsilon$ and nonpositive on
$\partial H_\varepsilon$, by \eqref{boundaries}, and at infinity. Thus, 
the maximum principle in
$H_\varepsilon$, Lemma~\ref{mpnarrow}, 
leads to $v-Su\leq 0$ in $H_\varepsilon$ and hence, by \eqref{full}, also
$$
v-Su\leq 0\quad\text{ in } \overline{\Omega}.
$$

We deduce that $S= w(x_0)=\sup_{\overline{\Omega}_\varep} w= \sup_{\overline{\Omega}} w$,
and thus the point $x_0\in\Omega$ obtained before is an interior maximum of $w$.
Now, note that 
\begin{eqnarray*}
\textrm{div}(u^2\nabla w)& = &\textrm{div} (\nabla v  \; u-v \, \nabla u)=
\Delta v  \;  u - v  \, \Delta  u = Lv  \;   u - v  \, L  u\\
& \geq & - v  \, L  u\quad \text{ in } \Omega. 
\end{eqnarray*}
Hence
\begin{equation}\label{ineqOm}
\Delta w + 2 u^{-1} \nabla  u  \, \nabla w +  u^{-1}L  u \; w \geq 0
\quad \text{ in } \Omega.
\end{equation}
But at the interior point $x_0\in \Omega$ of maximum of $w$, we have
\begin{eqnarray*}
& (\Delta w + 2 u^{-1} \nabla  u  \, \nabla w +  u^{-1}L  u \; w)(x_0) \leq \\
& \hspace{3cm} \leq ( u^{-1}L  u  \; w)(x_0) = S  u^{-1}(x_0) L u (x_0) <0
\end{eqnarray*}
by \eqref{stricL}, a contradiction with \eqref{ineqOm}. Thus, \eqref{contrad}
does not hold.
We conclude $\sup_{\overline{\Omega}_\varep} w \leq 0$ and hence $v\leq 0$ in 
$\overline{\Omega}_\varep$.

Finally, arguing for $v$ exactly as done before for $v-Su$, 
we deduce that $v \leq 0$ on $\partial H_\varep$. Then,
the maximum principle in
$H_\varepsilon$ leads to
$v\leq 0$ in $H_\varep$, and thus also in all $\Omega$ by \eqref{full}.
\end{proof}

The domain ${\mathcal N}_\varepsilon=\{t<s<t+\varep\}$ is a ``narrow'' domain in the sense
of \cite{BNV}, and thus Lemma \ref{mpnarrow} follows from
a very general maximum principle in ``narrow" domains
due to Berestycki, Nirenberg, and Varadhan~\cite{BNV}.
However, for completeness, below we give two different simple proofs 
of the lemma. First, let us explain what ``narrow'' means and why
the lemma follows from results of \cite{BNV, Ctop}.

Let $H\subset {\mathcal N}_\varep\subset \{t<s<t+\varep\}$ be an open set, 
as in Lemma~\ref{mpnarrow}.
Let $x$ be any point in $H$. It is simple to check that the distance from $x$ to the Simons cone
${\mathcal C}$ is given by the $z$ coordinate of $x$, i.e., by
$(s_x-t_x)/\sqrt{2}$; see Lemma~4.2 in \cite{CT1}. Hence, there exists 
a point $\overline{x}\in {\mathcal C}$  such that
$|x-\overline{x}|= (s_x-t_x)/\sqrt{2} < \varep/\sqrt{2}< (3/4)\varep$. Thus
$$
B_{\varep/4}(\overline{x})\setminus {\mathcal O} \subset
B_{\varep/4}(\overline{x})\setminus H \subset
B_{\varep}(x)\setminus H,
$$
and hence, since $\overline{x}\in {\mathcal C}$,
\begin{equation}\label{measout}
\begin{array}{rl}
2^{-1-4m} |B_{\varep}(x)| =& (1/2) |B_{\varep/4}(\overline{x})| =
|B_{\varep/4}(\overline{x})\setminus {\mathcal O}|
\vspace{1.5mm} \\
 \leq & |B_{\varep}(x)\setminus H|.
\end{array}
\end{equation}
Lemma~\ref{mpnarrow} now follows from Definitions~2.2 and 5.1 and
Theorem~5.2(i) of \cite{Ctop} ---a slightly more general version than
the result of \cite{BNV} to include unbounded domains. The specific
dependence \eqref{smallep} is the same as the one obtained in the proof 
of Theorem~5.2(i) of \cite{Ctop}.

Nevertheless, for completeness we present next two proofs 
of Lemma \ref{mpnarrow}. The first one follows
the proof in \cite{Ctop} and it was found by the author in \cite{Cabp}.
Replacing its technical tools (the mean value inequality for superharmonic
functions used below by the Krylov-Safonov weak Harnack inequality),
it applies to general  ``narrow'' domains and 
to operators in non-divergence form with bounded measurable coefficients.
Instead, our second proof will use strongly the specific ``shape'' of the domain 
${\mathcal N}_\varep$.

\begin{proof}[First proof of Lemma \ref{mpnarrow}]
Let $H\subset {\mathcal N}_\varep\subset \{t<s<t+\varep\}$ be an open set and 
$v\in C^{2}(H)\cap C(\overline H)$ satisfy
$$
Lv:=\Delta v+ \tilde{c}(x)v \geq 0 \,\hbox{ in } H, \,\, v\le 0 \hbox{ on
} \partial H,\hbox{ and } \limsup_{x\in H, |x|\to\infty}
v(x)\le 0.
$$
Arguing by contradiction, assume that $\sup_H v >0$. It follows that
the supremum of $v$ is achieved at some point $x_0 \in H$:
$$
\sup_H v = v(x_0) > 0.
$$

Let 
$$
K:=\|\tilde{c}_+\|_{L^\infty (H)}
$$
and 
$$ 
\phi (x) := (4m)^{-1} K v(x_0) (\varep^2-|x-x_0|^2)  \quad\text{ for } x\in 
\R^{2m}.
$$
Consider now the open set $H\cap\{v>0\}$. We have 
$$
-\Delta v \leq \tilde{c} v \leq \|\tilde{c}_+\|_{L^\infty (H)} v=Kv\leq K v(x_0)
=-\Delta \phi \quad\text{ in } H\cap\{v>0\}.
$$
Thus, $v-\phi$ is subharmonic in $B_\varep(x_0)\cap \left( H\cap\{v>0\}\right)$.
In addition, on $B_\varep(x_0)\cap \partial \left( H\cap\{v>0\}\right)$ 
we have $v-\phi \leq v \leq 0$. Thus, its positive part $(v-\phi)_+$, extended to be zero
in $B_\varep(x_0)\setminus \left( H\cap\{v>0\}\right)$,
is a continuous function which is subharmonic in the viscosity sense
(or in the distributional sense)
in $B_\varep(x_0)$.

We apply to $w:= v(x_0)-(v-\phi)_+$ the mean value inequality in the ball $B_\varep(x_0)$
for superharmonic functions in the viscosity (or distributional) sense.
Note also that $w> 0$ in $B_\varep(x_0)$
and recall the ``narrowness'' condition
\eqref{measout} with  $x=x_0\in H$. We have
\begin{eqnarray*}
2^{-1-4m} v(x_0) & \leq & \frac{|B_{\varep}(x_0)\setminus \left( H\cap\{v>0\}\right)|}
{|B_{\varep}(x_0)|} v(x_0)\\
&=& \frac{1}{|B_{\varep}(x_0)|} \int_{B_{\varep}(x_0)\setminus \left( H\cap\{v>0\}\right)} w \\
& \leq & \frac{1}{|B_{\varep}(x_0)|} \int_{B_{\varep}(x_0)} w \leq w(x_0)\\
& = & v(x_0)-(v(x_0)-\phi(x_0))_+ \\
&\leq &  \phi(x_0) = (4m)^{-1} \varep^2 K v(x_0).
\end{eqnarray*}
Thus, since we assumed $\sup_H v = v(x_0) > 0$, we get a contradiction 
whenever $(4m)^{-1}  \varep^2 K < 2^{-1-4m}$.
\end{proof}

The rest of this section is devoted to give another simple proof of
Lemma~\ref{mpnarrow}. In contrast with the previous one, the following proof is based on the
specific form of the domain ${\mathcal N}_\varep$. 

To give the proof,
we first need to establish the following easy result.

\begin{lemma}\label{mpsuper}
Let $H$ be an open set of $\R^n$, $\tilde{c}\in C(H)$, and $L=\Delta + \tilde{c}(x)$.
Assume that there exists a
function $\phi \in C(\overline H)$ (not necessarily bounded above)
such that 
\begin{equation*}
\phi \geq \delta >0 \quad\text{ in } \overline H
\end{equation*}
for some constant $\delta$. Assume also that there exists an open set $A\subset  H$
such that $\phi \in C^{2} (A)$,
\begin{equation}\label{strictsup}
L \phi <0 \quad \text{ in } A,
\end{equation}
and 
\begin{equation}\label{hessian}
\liminf_{\xi\to 0} \frac{\phi (x_0+\xi)+\phi (x_0-\xi)-2\phi (x_0)}{|\xi|^2} = -\infty
\quad\text{ for all } x_0\in  H\setminus A.
\end{equation}

Then, the maximum principle holds for $L$ in $ H$.
\end{lemma}

Even if it could be relaxed, note the strict inequality in \eqref{strictsup}.

Condition \eqref{hessian} prevents the function $\phi$ to be ``touched by
below'' at the point $x_0$ by a $C^2$ function (see the proof of the lemma for details).
As an example, the function $\phi(x)=-|x|$ satisfies \eqref{hessian} at $x_0=0$.
Another example appearing in applications is the distance function to a given point $p$
in a Riemannian manifold; it satisfies \eqref{hessian} at points $x_0$ in the cut locus
of $p$ (see \cite{Cnon}). It also occurs with the distance to the boundary 
$\partial H$ in an open
set $H$ of $\R^n$ at a cut point $x_0$ in $ H$ (see \cite{LN}).

\begin{proof}[Proof of Lemma \ref{mpsuper}]
Let $H\subset \R^n$ be an open set and 
$v\in C^{2}(H)\cap C(\overline H)$ satisfy
$$
Lv=\Delta v+ \tilde{c}(x)v \geq 0 \,\hbox{ in } H, \,\, v\le 0 \hbox{ on
} \partial H,\hbox{ and } \limsup_{x\in H, |x|\to\infty}
v(x)\le 0.
$$
Consider the function
$$
w:= \frac{v}{\phi},
$$
with $\phi$ as in Lemma \ref{mpsuper}. We have that $w$ is a continuous function in 
$\overline  H$ satisfying $w\le 0$ on $\partial H$ and 
$\limsup_{x\in H, |x|\to\infty} w(x)\le 0$.
Thus, $w$ is bounded above. 

Arguing by contradiction, assume that
$S:=\sup_ H w >0$. This supremum will be achieved at some point $x_0\in { H}$,
by the nonpositiveness of the limsup of $w$ at infinity.

We claim that $x_0\in A$. Indeed, we have that
$$
v\leq S \phi \quad\text{in } A \quad\text{ and }\quad  v(x_0)= S \phi(x_0). 
$$ 
It follows that the liminf for $\phi$ in \eqref{hessian} 
is greater than or equal to the same liminf for $S^{-1} v$, which is finite 
since $v\in C^2( H)$. By \eqref{hessian}, we conclude that $x_0\in A$.

Now, $v$, $\phi$, and $w$ are $C^2$ in $A$ and we have
\begin{eqnarray*}
\textrm{div}(\phi^2\nabla w)& = &\textrm{div} (\nabla v  \; \phi-v \, \nabla\phi)=
\Delta v  \; \phi - v  \, \Delta \phi = Lv  \;  \phi - v  \, L \phi\\
& \geq & - v  \, L \phi. 
\end{eqnarray*}
Hence
\begin{equation}\label{ineqA}
\Delta w + 2\phi^{-1} \nabla \phi  \, \nabla w + \phi^{-1}L \phi \; w \geq 0
\quad \text{ in } A.
\end{equation}
But at the point $x_0\in A$ of maximum of $w$, we have
\begin{eqnarray*}
& (\Delta w + 2\phi^{-1} \nabla \phi  \, \nabla w + \phi^{-1}L \phi \; w)(x_0) \leq \\
& \hspace{3cm} \leq (\phi^{-1}L \phi  \; w)(x_0) = S \phi^{-1}(x_0) L \phi (x_0) <0
\end{eqnarray*}
by \eqref{strictsup}, a contradiction with \eqref{ineqA}.

Thus, $\sup_H w \leq 0$ and hence $v\leq 0$ in $ H$.
\end{proof}

We can now give the second proof of the maximum principle in ${\mathcal N}_\varepsilon$.

\begin{proof}[Second proof of Lemma \ref{mpnarrow}]
Assume that
$$
3\varep^2 \|\tilde{c}_+\|_{L^\infty (H)} <1.
$$

We apply Lemma~\ref{mpsuper} with the choice
\begin{align}\label{expst}
\phi(x)= \phi (z) & := (z+\varepsilon)(3\varepsilon-z)=3\varepsilon^2+
2\varepsilon z- z^2\nonumber \\
& = 3\varepsilon^2+ \frac{2\varepsilon}{\sqrt 2} (s-t)- \frac{s^2+t^2-2st}{2}.
\end{align}
Note that $0<z<\varepsilon /\sqrt{2}<\varepsilon$ in ${\mathcal N}_\varepsilon$, and thus
$$
2\varepsilon^2\leq \phi\leq 6\varepsilon^2 \quad\text{in } {\mathcal N}_\varepsilon.
$$

For the set $A$ in Lemma~\ref{mpsuper} we choose
$$
A=H\cap \{ 0< t<s<t+\varepsilon\},
$$ 
and thus 
$$ 
H\setminus A\subset \{ t=0 \text{ and } 0<s<\varepsilon\}.
$$
Given a point $x_0\in H\setminus A$, since $x_0$ is a point with the $t$ coordinate
$t_0=0$ and with the $s$ coordinate $0<s_0<\varepsilon$, 
\eqref{expst} shows that in a neighborhood of $x_0$ 
the function $\phi$ is equal to a smooth function plus
$$
(-\sqrt{2}\varepsilon +s)t.
$$
Since $-\sqrt{2}\varepsilon + s_0< -\sqrt{2}\varepsilon +\varepsilon <0$,
considering second order incremental quotients in the $t$ variable, we see that
the liminf in \eqref{hessian} for this function at the point
$x_0$ is equal to $-\infty$. Thus, the same holds for~$\phi$.

Next, we have that $\phi\in C^2(A)$ and, in $A$,
$\phi_z=2\varepsilon-2z\geq 0$ and $\phi_{zz}=-2$.
Using expression \eqref{eqyz} to compute the Laplacian, we have
\begin{equation*}
\Delta \phi= \phi_{zz}-\frac{2(m-1)}{y^2-z^2}z\phi_z\quad\text{ in } A .
\end{equation*}
Hence,
$$
\Delta \phi +\tilde{c}\phi\leq \phi_{zz}+\tilde{c}\phi\leq -2+ 6 \varepsilon^2 
\|\tilde{c}_+\|_{L^\infty (H)} < 0 
\quad\text{in } A.
$$
This finishes the proof.
\end{proof}

\section{Uniqueness of saddle-shaped solution}

In this section we prove our uniqueness result, Theorem~\ref{unique}.
We use the maximum principle of the previous section and also the following
simple result.

\begin{lemma}\label{min}
Assume that $f$ satisfies \eqref{hypf} and that $u_1$ and $u_2$ are two
saddle-shaped solutions of \eqref{eq}, where $2m\geq 2$. Then,
there exists a saddle-shaped solution $u$ of 
\eqref{eq} such that
\begin{equation}\label{min2}
u \leq u_1 \quad \text{and}\quad u\leq u_2 \quad \text{ in } {\mathcal{O}}=\{s>t\}.
\end{equation}
\end{lemma}

This result follows from a more general one: Proposition~3.8 of~\cite{CT2} 
on the existence of a minimal saddle-shaped solution, i.e.,
smaller than or equal to any other saddle-shaped solution in ${\mathcal{O}}$.
However, the statement of Lemma~\ref{min} suffices for our purposes here and,
for completeness, we give next a simple proof of it.

\begin{proof}[Proof of Lemma \ref{min}]
Let
$$
w:=\min\{u_1,u_2\}  \quad \text{ in } \overline{\mathcal{O}},
$$
an $H^1$ function locally in $\overline{\mathcal{O}}$ and positive in ${\mathcal{O}}$. 

For $R>0$, consider the problem
\begin{equation}
\label{probR}
\left\{\begin{array}{rcll}
-\Delta u_R & = & f(u_R) & \text{ in } {\mathcal O}\cap B_R(0)\\
u_R & = & w & \text{ on } \partial ({\mathcal O}\cap B_R(0)).
\end{array}\right.
\end{equation}
By its definition, $w$ is a weak supersolution of \eqref{probR},
while $0$ is clearly a subsolution. As a consequence, there exists a weak solution
$u_R$ of \eqref{probR} with $0\leq u_R\leq w$.
It can be taken to be a minimizer of the energy
functional ${\mathcal E}(\cdot,{\mathcal O}\cap B_R(0))$, defined by \eqref{energia},
in the convex set
\begin{eqnarray*}
& \hspace{-3.6cm} K_w:=\Big\{v\in H^1({\mathcal O}\cap B_R(0))\, :\, v=v(s,t) \text{ a.e., } \\
& \hspace{1.4cm} 
0\leq v\leq w \text{ in } 
{\mathcal O}\cap B_R(0),
\text{ and } v\equiv w \text{ on } \partial({\mathcal O}\cap B_R(0))\Big\}
\end{eqnarray*}
of functions of $s$ and $t$ only. Note that $K_w$ is weakly
closed in $H^1({\mathcal O}\cap B_R(0))$. For more details, see the proofs of 
Theorem~1.3 in~\cite{CT1} and of Theorem~2.4 in~\cite{St}. The set
${\mathcal O}\cap B_R(0)$ not being Lipschitz at the origin (when $2m\geq 4$)
may be avoided removing from it
a small ball $B_\varep (0)$, minimizing here, and then letting $\varep\to 0$. 

Since $0$
is not a weak solution of \eqref{probR}, the strong maximum principle leads to
$$
0<u_R=u_R(s,t)\leq w=w(s,t) \quad \text{ in } {\mathcal O}\cap B_R(0).
$$

Next, by elliptic estimates and the Arzela-Ascoli theorem (see \cite{CT1,CT2} for more
details), the limit as $R\to\infty$ of $u_{R}$ 
exists (up to subsequences) in every compact set of $\overline {\mathcal O}$. 
We obtain a solution $u$ of $-\Delta u= f(u)$ in
${\mathcal O}=\{s>t\}$ such that $u=0$ on ${\mathcal C}$ and
$0\leq u \leq w$ in ${\mathcal O}$.
Reflecting $u=u(s,t)$ to be odd with respect to the Simons cone,
we obtain a solution $u=u(s,t)$ of \eqref{eq} in all of $\R^{2m}$
satisfying \eqref{min2}.

To finish the proof it remains to show that $u>0$ in ${\mathcal O}$. This
will ensure that $u$ is a saddle-shaped solution. 
We use the argument in \eqref{possuper}; it gives that $u_R>0$ is a positive supersolution of the 
linearized operator $\Delta + f'(u_R)$ in ${\mathcal O}\cap B_R(0)$.
As a consequence (see section~2) the maximum principle holds for this operator
in compact subdomains of ${\mathcal O}\cap B_R(0)$, and hence its first Dirichlet eigenvalue
in these domains is positive. We deduce, by Rayleigh criterion, that $Q_{u_R}(\xi)\geq 0$ 
for every smooth function $\xi$ with
compact support in ${\mathcal O}\cap B_R(0)$ ---recall that $Q_{u_R}$
is defined in \eqref{stable}. The conclusion $Q_{u_R}(\xi)\geq 0$ could also
been verified in a different, very simple way. Simply use that 
$u_R$ is a positive supersolution of the 
linearized operator and the integration by parts argument preceding \eqref{stabparts}
in section~5.

Now, letting $R\to\infty$, we are led to 
$Q_{u}(\xi)\geq 0$ for all smooth functions $\xi$ with
compact support in ${\mathcal O}$. This would be a contradiction with
$u\equiv 0$ in ${\mathcal O}$, since in such case $f'(u)=f'(0)$ is a positive constant
and hence $-\Delta-f'(0)$ is not a nonnegative operator in balls
of ${\mathcal O}$ with sufficiently large radius.

Therefore, $u\geq 0$ and $u\not \equiv 0$ in ${\mathcal O}$. It follows that
$u>0$ in ${\mathcal O}$, by the strong maximum principle.
\end{proof}

The existence of the solution $u_R$ in the above proof could also
be shown by the monotone iteration procedure; see \cite{CT2}.
On the other hand, the fact that $u>0$ in ${\mathcal O}$ could also
be proved placing an explicit subsolution below all $u_R$; see 
Remark~3.6 in \cite{CT2}.

We finish this section proving our uniqueness result.

\begin{proof}[Proof of Theorem \ref{unique}]
Let $u_1$ and $u_2$ be two saddle-shaped solutions of \eqref{eq}. 
Let $u$ be the saddle-shaped solution of Lemma~\ref{min}.
Consider the difference $v:= u_i-u$ for $i=1$ and $i=2$.
We have that
$$
-\Delta (u_i-u) = f(u_i)-f(u) \leq f'(u) (u_i-u)\quad\text{ in } {\mathcal{O}}=\{s>t\},
$$
since in this set $u\leq u_i$ and $f$ is concave in $(0,1)$. Thus,
$$
L_{u} (u_i-u) := \{ \Delta + f'(u(x)) \} (u_i-u) \geq 0 \quad\text{ in }
{\mathcal{O}}=\{s>t\}.
$$
In addition, we have that $u_i- u\equiv 0$ on 
${\mathcal{C}}=\partial{\mathcal{O}}$ and 
$$
\limsup_{x\in{\mathcal O}, |x|\to\infty} (u_i-u)(x)= 0
$$ 
by the asymptotic result \eqref{unif} applied to both
$u_i$ and $u$. 

To the saddle-shaped solution $u$, we apply 
the maximum principle of Proposition~\ref{mpsaddle} ---a particular case of 
Proposition~\ref{mpholds} proven in the previous section.
We obtain that the maximum principle holds for 
$L_{u} = \Delta + f'(u(x))$ in
${\mathcal{O}}$. Since $v= u_i-u$ satisfies hypotheses \eqref{hypintroMP} by the 
above facts, we deduce
$u_i-u\leq 0$ in ${\mathcal{O}}$. Thus, 
by \eqref{min2}, $u_i-u\equiv 0$ in ${\mathcal{O}}$. Since this holds for
both $i=1$ and $i=2$, we deduce $u_1 \equiv u_2$, that is, uniqueness.
\end{proof}

\section{Monotonicity and convexity properties}

We start this section with some regularity issues needed in the subsequent.
Recall that we assume that $f\in C^{2,\alpha}$ for some $\alpha\in (0,1)$.
Let $u=u(x)$ be a bounded solution of \eqref{eq}. Since $f(u)\in L^\infty$,
it is also an $L^p$ function for all $1<p<\infty$ in every ball of radius $2$,
with a uniform bound on its $L^p$-norm in such balls. Thus,
$u\in W^{2,p}\subset C^{1,\alpha}$ (if $p$ is taken large enough) with uniform
bounds in every ball of radius~$1$ (i.e., with half the radius of the previous ones). 
Now, we have
$-\Delta u_{x_i}=f'(u) u_{x_i} \in  C^{\alpha}$ for all indexes~$i$, and hence
$ u_{x_i} \in C^{2,\alpha}$. But now we know
$-\Delta u_{x_i}=f'(u) u_{x_i} \in  C^{1,\alpha}$, and thus $ u_{x_i} \in C^{3,\alpha}$.
That is, we have 
\begin{equation}\label{bounded4}
u\in C^{4,\alpha} (\R^{2m}) \quad\text{and}\quad D^ku \in L^\infty(\R^{2m}) \ \text{ if }
0\leq |k|\leq 4.
\end{equation}

Assume now that $u=u(x)=u(s,t)$ is a bounded solution that depends only on $s$ and $t$
---as in the case of saddle-shaped solutions. For $\tilde{s}\in \R$ and 
$\tilde{t}\in \R$, let
$$
\tilde{u}(\tilde{s},\tilde{t}):=u(\tilde{s}, x_2=0,\ldots,x_m=0, \tilde{t},
x_{m+2}=0,\ldots,x_{2m}=0).
$$
Since $u\in C^4 (\R^{2m})$, we deduce that $\tilde{u}\in C^4 (\R^{2})$
and hence $u=u(s,t)$ is also a $C^4$ function of the variables
$s\geq 0$ and $t\geq 0$. Furthermore,  $u=u(s,t)$ is the restriction
to $(s,t)\in [0,\infty)\times [0,\infty)$ of a $C^4(\R^2)$ function
$\tilde{u}$ which is even in $s$ and in $t$. In particular we have
\begin{equation}\label{deriv10}
u_s=0 \quad\text{in } \{s=0\} \qquad \text{and} \qquad 
u_t=0 \quad\text{in } \{t=0\}.
\end{equation}
As a consequence,
\begin{equation}\label{deriv20}
u_{st}\in C^2(\R^{2m}) \quad\text{and}\quad u_{st}=0 \; \text{ in } \{st=0\}.
\end{equation}

To establish the statement $u_{st}>0$ in $\{s>t>0\}$ of Proposition~\ref{monot},
we need the following asymptotic result.

\begin{lemma}\label{lemmast}
Assume that $f$ satisfies \eqref{hypf}.
Let $u$ be the saddle-shaped solution 
of $-\Delta u= f(u)$ in~$\R^{2m}$, where $2m\geq 2$.

Then,
\begin{equation}
\label{unifst}
\vectornorm{D^2_{(s,t)} (u-U)}_{L^{\infty}(\{st>0, s^2+t^2 \geq R^2\})}
\longrightarrow 0\quad
\text{as }\, R\rightarrow\infty ,
\end{equation}
where $U$ is defined in \eqref{defU}.
\end{lemma}

Recall that $U$ is a Lipschitz function in all of $\R^{2m}$,
but it is not $C^1$ at $\{st=0\}$.
It is therefore important to take the sup-norm 
of $D^2_{(s,t)} (u-U)$ as a function of the two variables $s$ and $t$, in
$\{st>0, s^2+t^2 \geq R^2\}$ ---which does not contain $\{st=0\}$.

\begin{proof}[Proof of Lemma \ref{lemmast}]
We follow the proof of Theorem~1.6 of \cite{CT2}. It argues by contradicting
\eqref{unif}
---here by contradicting \eqref{unifst}---
and in this way obtaining a sequence of points $\{x_k\}$, with $|x_k|\to\infty$,
for which one of these asymptotics does not hold. By odd symmetry, and taking a subsequence,
one may assume that $\{x_k\}\subset{\mathcal O}$. 

Next, one translates the solution
$u$ to be centered now at $x_k$,
and uses a translation and compactness argument; compactness comes from a priori estimates
and the Arzela-Ascoli theorem. 
The translated solutions converge to 
a solution $v$ in all of $\R^{2m}$ in the $C^4$ uniform convergence in compact sets, 
since any uniformly bounded sequence of solutions is uniformly bounded in
$C^4$ on every compact set, as shown above. 
The points $x_k$ in the proof satisfy
$|x_k|\to\infty$ and now, in addition,
$s_kt_k>0$ ---since we are contradicting the 
$L^{\infty}(\{st>0, s^2+t^2 \geq R^2\})$ convergence.

Now, in case~1 of the proof we have that the distances to the Simons cone
$|z_k|=z_k=(s_k-t_k)/\sqrt{2}\to\infty$ and thus the limiting solution
$v$ is defined in all $\R^{2m}$ and satisfies $0\leq v\leq 1$. By stability
of $u$ in ${\mathcal O}$ we deduce the stability of $v$ in $\R^{2m}$.
Thus, $v\not\equiv 0$ and therefore a Liouville-type theorem 
of Aronson and Weinberger \cite{AW} (see also \cite{BHN} for a more
general version, and \cite{CT2} for the statements) 
guarantees that $v\equiv 1$. Thus $\| D^2_{(s,t)}u(s_k,t_k) \|\rightarrow 0$,
and since $\| D^2_{(s,t)}U(s_k,t_k) \|
=|u_0''(z_k)|\rightarrow 0$ because $z_k\to +\infty$, the proof arrives at a
contradiction. 

Finally, in case~2 of the proof, the points $x_k$
remain at a finite distance of the Simons cone. Since the curvatures of
a cone tend to zero at infinity, in this case the limiting solution
$v$ is nonnegative in a certain limiting half-space $\R_+^{2m}$
and $v$ vanishes at its boundary. By stability again, $v\not \equiv 0$ and hence $v>0$
in the half-space. Then, a Liouville theorem of Angenent \cite{An} 
(see also \cite{CT2} for the statement) gives that $v$ is the 1D solution $u_0$
depending only on the Euclidean variable orthogonal to the boundary of the
half-space. Since $\{z_k\}$ are the distances to the cone and remain bounded, 
in the limit
this solution agrees with $u_0(z)=U(x)$. Hence, the full Hessian
$D^2_x (u-U)(x_k)$ tends to zero.
\end{proof}

We can now give the

\begin{proof}[Proof of Proposition \ref{monot}]
Let $u$ be the saddle-shaped solution of \eqref{eq}. 
Differentiating \eqref{eqst} with respect to $s$ and $t$ we get
\begin{equation}\label{equs}
\Delta u_s +f'(u)u_s -\frac{m-1}{s^2}u_s =0
\quad \text{ in } \R^{2m}\setminus\{s=0\}
\end{equation}
and
\begin{equation}\label{equt}
\Delta u_t +f'(u)u_t 
-\frac{m-1}{t^2}u_t =0 \quad \text{ in } \R^{2m}\setminus\{t=0\}.
\end{equation}

Taking into account \eqref{equs}, 
we apply the maximum principle of 
Proposition~\ref{mpholds} to the 
function $u_s$ in $\Omega:=\{s>t\}={\mathcal O}\subset\R^{2m}$ 
with $c(x):=-(m-1)s^{-2}$, 
a negative continuous function in $\{s>t\}$.
Recall that $u_s$ is $C^{2}$ in all $\R^{2m}$ 
and note that it satisfies $u_s\geq 0$ on $\partial {\mathcal O}=\{s=t\}$
since $u\equiv 0$ on $\{s=t\}$ and $u>0$ in $\{s>t\}$. Furthermore, we have
$\limsup_{x\in{\mathcal O}, |x|\to\infty} u_s(x)\ge 0$, by the
asymptotic result \eqref{unif} and since $U_s(x)=u_0'((s-t)/\sqrt{2})/\sqrt{2}
\geq 0$. We deduce that 
\begin{equation}\label{positus}
u_s\geq 0 \quad\text{ in } {\mathcal O}=\{s>t\}.
\end{equation}

Next, we apply Proposition~\ref{mpholds} in a different subdomain
of  ${\mathcal O}$. We apply it to the equation \eqref{equt} and 
the function $u_t$ in $\Omega:=\{s>t>0\}\subset\R^{2m}$, with $c(x):=-(m-1)t^{-2}$, 
a negative
continuous function in $\{s>t>0\}$.
Note that $u_t\leq 0$ on $\partial \{s>t>0\}=\{s=t\}\cup\{t=0\}$;
here we use \eqref{deriv10}.
We also have
$\limsup_{x\in\{s>t>0\}, |x|\to\infty} u_t(x)\leq 0$, by the
asymptotic result \eqref{unif} and since $U_t(x)=-u_0'((s-t)/\sqrt{2})/\sqrt{2}
\leq 0$. We deduce that 
\begin{equation}\label{positut}
u_t\leq 0 \quad\text{ in } \{s>t>0\}.
\end{equation}

Since $u(s,t)=-u(t,s)$ in all $\R^{2m}$, \eqref{positus} and \eqref{positut}
lead to $u_s\geq 0$ in all $\R^{2m}$. This, the strong maximum principle, and
equation \eqref{equs} finally give
\begin{equation}\label{strpositus}
u_s > 0 \quad\text{ in } \R^{2m}\setminus\{s=0\}.
\end{equation}
Symmetrically, we have
\begin{equation}\label{strpositut}
-u_t > 0 \quad\text{ in } \R^{2m}\setminus\{t=0\}.
\end{equation}
In particular, statement \eqref{signut} of the proposition is now proved.

To establish \eqref{signuy}, using 
$\partial_y=(\partial_s+\partial_t)/\sqrt{2}$, we obtain
\begin{eqnarray*}
\Delta u_y  +f'(u)u_y & =
 & \frac{m-1}{\sqrt{2}}\left(\frac{u_s}{s^2}+
\frac{u_t}{t^2}\right)\\
&  & \hspace{-18mm} =\, \frac{m-1}{s^2}u_y + 
\frac{(m-1)(s^2-t^2)}{\sqrt{2}s^2t^2}u_t
\quad \text{ in } \R^{2m}\setminus\{st=0\}.
\end{eqnarray*}
Thus, by \eqref{strpositut}, we deduce
\begin{equation}\label{strequy}
\Delta u_y  +f'(u)u_y -\frac{m-1}{s^2}u_y \leq 0
\quad \text{ in } {\mathcal O}=\{s>t\}.
\end{equation}
We apply Proposition~\ref{mpholds} to the 
function $u_y$ in $\Omega:={\mathcal O}=\{s>t\}\subset\R^{2m}$ with $c(x):=-(m-1)s^{-2}$.
Note that $u_y\equiv 0$ on $\partial {\mathcal O}=\{s=t\}=\{z=0\}$. 
Furthermore, we have
$\limsup_{x\in{\mathcal O}, |x|\to\infty} u_y(x) = 0$, by the
asymptotic result \eqref{unif} and since $U_y\equiv 0$ in all $\R^{2m}$. 
We deduce that $u_y\geq 0$ in ${\mathcal O}=\{s>t\}$. 
This, the strong maximum principle, and \eqref{strequy} give
$u_y>0$ in ${\mathcal O}=\{s>t\}$, i.e., \eqref{signuy} of the proposition.

It remains to establish \eqref{second}. Differentiating \eqref{equs} with respect
to~$t$ and recalling the expression of the Laplacian in 
$(s,t)$ variables, we obtain
\begin{eqnarray}
\label{equst}
& & \hspace{-20mm}
\Delta u_{st}  +f'(u) u_{st} 
-(m-1) \left(\frac{1}{s^2}+\frac{1}{t^2}\right)  u_{st}  
\nonumber
\\ & = & -f''(u) u_s u_ t 
\nonumber
\\ & \leq  & 0 
\qquad\hspace{1cm} \text{ in } \{s>t>0\}.
\end{eqnarray}

We apply Proposition~\ref{mpholds} to this inequality and to
the $C^2(\R^{2m})$ function $u_{st}$ ---recall \eqref{deriv20}--- in the domain $\Omega:=\{s>t>0\}\subset\R^{2m}$, 
with $c(x):=-(m-1)(s^{-2}+t^{-2})$, a negative
continuous function in $\{s>t>0\}$.
Note that $\partial \{s>t>0\}=\{s=t\}\cup\{t=0\}$ and that $u_{st}=0$ on $\{t=0\}$
by \eqref{deriv20}. In addition, since $u=0$ on $\{s=t\}=\{z=0\}$ we have 
$u_{yy}=0$ on $\{s=t\}=\{z=0\}$. Since $u$ is odd with respect to $z$, we also have
$u_{zz}=0$ on $\{s=t\}=\{z=0\}$. Thus, since
$$
u_{st}=\frac{1}{2}(u_{yy}-u_{zz}),
$$
we deduce that $u_{st}=0$ on $\{s=t\}$. Finally, note that
$$
\limsup_{x\in\{s>t>0\}, |x|\to\infty} u_{st}(x)\geq 0, 
$$
by the asymptotic result \eqref{unifst} and since $U_{st}(x)=(1/2)(U_{yy}-U_{zz})(z)
=-U_{zz}(z)/2=-u_0''(z)/2 =f(u_0(z))/2 \geq 0$ in $\{s>t\}=\{z>0\}$. 
Proposition~\ref{mpholds}
leads to $u_{st}\geq 0$ in $\{s>t>0\}$.
{From} this, \eqref{equst}, and the strong maximum principle, we conclude
the strict sign for $u_{st}$ in $\{s>t>0\}$, as stated in \eqref{second}.
\end{proof}

\section{The supersolution of the linearized equation}

We end up establishing our stability result.

\begin{proof}[Proof of Theorem \ref{stab}]
Let $u$ be the saddle-shaped solution of \eqref{eq} in $\R^{2m}$.
Recall that since $2m\geq 14$, we can take $b>0$ satisfying \eqref{ineqb},
or equivalently \eqref{rangeb}. Let
$$
\varphi := t^{-b}u_s -s^{-b}u_t,
$$
a $C^2$ function in $\{st> 0\}$. By \eqref{strpositus} and \eqref{strpositut},
we have that 
\begin{equation}\label{posivar}
\varphi >0 \qquad\text{ in } \{st> 0\}. 
\end{equation}
Now, since $u(t,s)=-u(s,t)$, one easily verifies that $\varphi (t,s)=\varphi (s,t)$,
i.e., $\varphi$
is even with respect to $z$. Thus $\{\Delta +f'(u)\}\varphi$ is also even with respect to~$z$,
and hence we only need to show that 
$\{\Delta +f'(u)\}\varphi\leq 0$ in $\{s> t >0\}$. 
{From} this we will deduce the same inequality in all $\{st> 0\}$
---as stated in the theorem. Then, at the end of the proof, we will show
that this easily leads to the stability of $u$ in all of $\R^{2m}$.

In $\{s>t>0\}$, we have
$$
\Delta t^{-b}= b(b-m+2)t^{-b-2}\quad\text{ and }\quad
\Delta s^{-b}= b(b-m+2)s^{-b-2}.
$$
Thus, using also \eqref{equs} and \eqref{equt}, in $\{s>t>0\}$
\begin{eqnarray*}
\Delta \varphi & = & b(b-m+2)t^{-b-2}u_s \\
& &  -f'(u)u_st^{-b}+ (m-1)s^{-2}u_st^{-b} -2bt^{-b-1}u_{st}\\
& & - b(b-m+2)s^{-b-2}u_t  \\
& & -\left\{-f'(u)u_t s^{-b} + (m-1)t^{-2}u_t s^{-b} -2bs^{-b-1}u_{st}\right\},
\end{eqnarray*}
and hence
\begin{eqnarray*}
\{\Delta + f'(u)\} \varphi & = & t^{-b} u_s \{ (m-1)s^{-2} + b(b-m+2) t^{-2} \} \\
& &  - s^{-b} u_t \{ (m-1)t^{-2} +  b(b-m+2) s^{-2} \} \\
& & +2b u_{st} \{ s^{-b-1} -t^{-b-1} \}.
\end{eqnarray*}
Now, using that, in $\{s>t>0\}$, $u_{st}>0$, $u_{y}>0$, and $-u_{t}>0$ 
(by Proposition~\ref{monot}),
and also the inequality \eqref{ineqb} for $b>0$, we arrive at
\begin{eqnarray*}
\{\Delta + f'(u)\} \varphi &\leq  & t^{-b} (u_s+u_t) \{ (m-1)s^{-2} + b(b-m+2) t^{-2} \} \\
& &  - s^{-b} u_t \{ (m-1)t^{-2} +  b(b-m+2) s^{-2} \} \\
& & -t^{-b} u_t \{ (m-1)s^{-2} + b(b-m+2) t^{-2} \}\\
& =  &  u_y\; \sqrt{2} t^{-b} \{ (m-1)s^{-2} + b(b-m+2) t^{-2} \} \\
& &  + (-u_t)\;  (m-1)(s^{-b} t^{-2}+ t^{-b} s^{-2}) \\ 
& &  + (-u_t)\; b(b-m+2) (s^{-2-b}+t^{-2-b}) \\ 
& \leq & u_y\;  \sqrt{2} t^{-b} (m-1)\{ s^{-2} - t^{-2} \} \\
& &  + (-u_t)\; (m-1) (s^{-b} t^{-2}+ t^{-b} s^{-2}-s^{-2-b}-t^{-2-b})  \\
&   \leq & (-u_t)\; (m-1) (s^{-b} t^{-2}+ t^{-b} s^{-2}-s^{-2-b}-t^{-2-b})
\end{eqnarray*}
in $\{s>t>0\}$. Finally, since in $\{s>t>0\}$ we have $-u_{t}>0$ and 
\begin{eqnarray*}
s^{-b} t^{-2}+ t^{-b} s^{-2}-s^{-2-b}-t^{-2-b} &=&
s^{-b} (t^{-2}-s^{-2}) + t^{-b} (s^{-2}-t^{-2})\\
&=& (s^{-b} - t^{-b})(t^{-2}-s^{-2}) \leq 0,
\end{eqnarray*}
we conclude $\{\Delta + f'(u)\} \varphi \leq 0$ in $\{s>t>0\}$. Hence,
by even symmetry in~$z$, also 
\begin{eqnarray}\label{superpf}
\{\Delta + f'(u)\} \varphi \leq 0 \qquad\text{ in } \R^{2m}\setminus\{st=0\}=\{st>0\}.
\end{eqnarray}

Next, using \eqref{posivar} and \eqref{superpf}, we can verify the stability
condition for any $C^1$ test function
$\xi = \xi(x)$ with compact support in $\{st>0\}$. Indeed, multiply \eqref{superpf}
by $\xi^2/\varphi$ and integrate by parts to get
\begin{eqnarray*}
\int_{\{st>0\}} f'(u)\,\xi^2 \, dx  
& = &\int_{\{st>0\}} f'(u) \varphi\, \frac{\xi^2}{\varphi} \, dx \\
& \leq & \int_{\{st>0\}} -\Delta \varphi \, \frac{\xi^2}{\varphi}  \, dx \\
& = &\int_{\{st>0\}} 
\nabla \varphi \, \nabla\xi\, \frac{2\xi}{\varphi}  \, dx 
-\int_{\{st>0\}} \frac{|\nabla \varphi|^2}{\varphi^2}\, \xi^2 \, dx .
\end{eqnarray*}
Now, using the Cauchy-Schwarz inequality, we are led to
\begin{equation}\label{stabparts}
\int_{\{st>0\}} f'(u)\,\xi^2 \, dx \leq \int_{\{st>0\}} |\nabla\xi|^2 \, dx .
\end{equation}

Finally, we need to prove this same inequality for every $C^1$ function
$\xi$ with compact support in a ball $B_{R_0}(0)\subset \R^{2m}$. For this, let 
$\eta_{\varepsilon}$ be a smooth function in $[0,\infty)$ with
$0\leq\eta\leq 1$, being identically $0$
in $[0,\varepsilon/2)$ and identically $1$
in $[\varepsilon,\infty)$. Since $\xi(x)\eta_{\varepsilon}(s)\eta_{\varepsilon}(t)$ 
is a $C^1$ function of $x$ with compact support in $\{s\geq\varepsilon/2, t\geq\varepsilon/2\}$,
the stability property just proven gives
$$
\int_{\R^{2m}} f'(u(x))\,\xi^2(x) \eta_{\varepsilon}^2(s)\eta_{\varepsilon}^2(t)\, dx 
\leq \int_{B_{R_0}(0)} 
|\nabla_x \left\{ \xi(x)\eta_{\varepsilon}(s)\eta_{\varepsilon}(t) \right\}
|^2 \, dx .
$$
We now compute all the terms in the right hand side of this inequality
and, using Cauchy-Schwarz, we see that to conclude
$$
\int_{\R^{2m}} f'(u)\,\xi^2 \, dx \leq \int_{\R^{2m}} |\nabla\xi|^2 \, dx 
$$
by letting $\varepsilon\to 0$, it is enough to use that 
\begin{eqnarray*}
\int_{B_{R_0}(0)} |\nabla_x \eta_{\varepsilon}(s)|^2\, dx 
& \leq  & \int_{\{s\leq \varepsilon, t \leq R_0\}} |\nabla_x \eta_{\varepsilon}(s)|^2\, dx \\
& \leq & \int_{\{s\leq \varepsilon, t \leq R_0\}} C\varepsilon^{-2} s^{m-1} t^{m-1}\, ds dt\\
& \leq & C \varepsilon^{m-2} R_0^m \longrightarrow 0 \qquad\text{as } 
\varepsilon\rightarrow 0 
\end{eqnarray*}
since $m\geq 3$ ---and the same for the integral of $|\nabla_x \eta_{\varepsilon}(t)|^2$.
This concludes the proof.
\end{proof}

\end{document}